\documentclass[12pt,a4paper]{amsart}

\usepackage{amsmath,amsthm,amssymb,latexsym,a4wide,tikz,multicol,tikz-cd,dsfont}
\usepackage{arydshln,multirow}
\usepackage{mathtools,stmaryrd,upgreek}
\usepackage{soul}
\usepackage{eucal}
\usepackage{amscd}
\usepackage{comment}
\usepackage{tikz}
\usetikzlibrary{positioning}
\usetikzlibrary{cd}
\usepackage[active]{srcltx}
\usepackage{enumerate}
\usepackage{changes}
\usepackage{cancel}
\usepackage{soul}

\DeclareSymbolFont{rsfscript}{OMS}{rsfs}{m}{n}
\DeclareSymbolFontAlphabet{\mathrsfs}{rsfscript}

\newtheorem{theorem}{Theorem} [section]
\newtheorem{lemma}[theorem]{Lemma}
\newtheorem{corollary}[theorem]{Corollary}
\newtheorem{proposition}[theorem]{Proposition}

\theoremstyle{definition}

\newtheorem{remark}[theorem]{Remark}

\numberwithin{equation}{section}

\def\softl{l\kern-0.3ex\raise0.1ex\hbox{'}\kern-0.3ex}  
\def\circledm{\protect\mathbin{\hbox
    {\protect$\bigcirc$\rlap{\kern-9.2pt\raise0pt\hbox
    {\protect$\mathtt{m}$}}}}}        
\def\malc{\circledm}            
\def\smallcircledm{\protect\mathbin{\hbox
    {\protect$\bigcirc$\rlap{\kern-6.9pt\raise0pt\hbox
    {\protect$\mathtt{m}$}}}}}

\def\rbl{{-\!\!\!-\!\!\!-\!\!\!-\!\!\!\!\rightarrow}}

\newcommand{\FG}{\operatorname{F\mathbf{G}}}
\newcommand{\FH}{\operatorname{F\mathbf{H}}}
\newcommand{\FSL}{\operatorname{F\mathbf{Sl}}}
\newcommand{\BR}{\operatorname{BR}}
\newcommand{\Ed}{\operatorname{Ed}}
\newcommand{\Sub}{\operatorname{Sub}}
\newcommand{\Cay}{\operatorname{Cay}}
\newcommand{\F}{F}
\newcommand{\M}{M}
\newcommand{\Msg}{M_{\mbox{\tiny sg}}}
\renewcommand{\P}{P}
\newcommand{\Q}{Q}
\newcommand{\f}{i}

\newcommand{\T}{\mathbb{T}}
\newcommand{\I}{\mathbb{I}}
\newcommand{\mx}[1]{#1^{\mathfrak{m}}}
\newcommand{\Imx}{\mathbb{I}\mathfrak{m}}

\newcommand{\bfFA}{\mathbf{FI}}
\newcommand{\bfPF}{\mathbf{PF}}
\newcommand{\GX}{\mathcal{G}_X}
\newcommand{\FX}{\mathcal{F}_X}
\newcommand{\eps}{\varepsilon}

\renewcommand{\phi}{\varphi}
\newcommand{\cR}{\mathrsfs{R}}
\newcommand{\cL}{\mathrsfs{L}}
\newcommand{\cD}{\mathrsfs{D}}
\newcommand{\cJ}{\mathrsfs{J}}

\title[$F$-inverse monoids]{$F$-inverse monoids as algebraic structures\\  in enriched signature}

\author{K.~Auinger, G.~Kudryavtseva and M.~B.~Szendrei}
\address{Fakult\"at f\"ur Mathematik, Universit\"at Wien, Oskar-Morgenstern-Platz 1, A-1090 Wien, Austria}
\email{karl.auinger@univie.ac.at}
\address{University of Ljubljana,
Faculty of Civil and Geodetic Engineering, Jamova cesta~2, SI-1000 Ljubljana, Slovenia; Institute of Mathematics, Physics and Mechanics, Jadranska ulica 19, SI-1000 Ljubljana, Slovenia}
\email{ganna.kudryavtseva@fgg.uni-lj.si, ganna.kudryavtseva@imfm.si}
\address{Bolyai Institute, University of Szeged, Aradi v\'ertan\'uk tere 1, H-6720 Szeged, Hungary; Alfr\'ed R\'enyi Institute of Mathematics, Hungarian Academy of Sciences, Re\'altanoda utca 13--15, H-1053 Budapest, Hungary}
\email{m.szendrei@math.u-szeged.hu}
\thanks{The second author was partially supported by ARRS grant P1-0288, the third author by the National Research, Development and Innovation Office, grants K115518 and K128042.}

\begin{document}

\begin{abstract} Every $F$-inverse monoid can be equipped with the unary operation $a\mapsto \mx{a}$ where $\mx{a}$ denotes the maximum element in the $\sigma$-class of $a$. In this enriched signature, the class of all $F$-inverse monoids forms a variety of algebraic structures. We describe universal objects in several classes of $F$-inverse monoids, in particular free $F$-inverse monoids.
More precisely, for every $X$-generated group $G$ we describe the initial object in the category of all $X$-generated $F$-inverse monoids $F$ for which $F/\sigma=G$.

\end{abstract}

\maketitle

\section{Introduction}

The origin of the notion of an $F$-inverse monoid lies in the theory of partially ordered semigroups, and
their class is one of the most studied subclasses of inverse semigroups, see Lawson \cite{Lawson} and Petrich \cite{Petrich}.
Their impact exceeds semigroup theory: they are useful, for instance, in the theories of partial group actions, see
Kellendonk and Lawson \cite{KL}, and of $C^*$-algebras, see Milan and Steinberg \cite{milan, steinberg}, Li and Norling \cite{li} and Starling \cite{starling}.
For undefined notions and unproven facts in context of inverse semigroups and general algebraic structures, the reader is referred to Lawson \cite{Lawson} and Burris--Sankappanavar \cite{Burris}, respectively.

An inverse monoid is $F$-inverse if each class of its minimum group congruence $\sigma$ has a maximum element with respect to the natural partial order.
In this case, the $\sigma$-class containing the identity element consists of all idempotents, and so
an $F$-inverse monoid is necessarily $E$-unitary.
In particular, free inverse monoids are $F$-inverse.
Within inverse semigroups, $F$-inverse monoids appear to be ubiquitous.
On the one hand, every $E$-unitary inverse monoid is embeddable in an $F$-inverse monoid, namely in its permissible hull,
and on the other hand, every inverse monoid has an $F$-inverse cover. The finitary version of the former statement is
an easy consequence since forming the permissible hull preserves finiteness.  The question whether each finite inverse monoid
has a finite $F$-inverse cover is a long-standing open problem.
When considering a possible computational attack on this problem, Kinyon observed that the class of all
$F$-inverse monoids forms a variety if the unary operation $a\mapsto \mx{a}$ where $\mx{a}$ is the maximum element in the $\sigma$-class of $a$ is added to the set of basic operations.
His announcement of this fact in \cite{Kinyon} at the International Conference on Semigroups (Lisbon, 2018) stimulated the
authors to find transparent models of free $F$-inverse monoids, from which the work on the present paper began.

Throughout the paper (from Section \ref{sect:F-inverse monoids} on), $F$-inverse monoids are understood in this extended signature.
The main result of the paper (Theorem \ref{thm:univF}) is proved in Section \ref{sect:F}.
For any $X$-generated group $G$, we consider an $F$-inverse monoid $\F(G)$ defined in a way extending the definition of
the Margolis--Meakin expansion $\M(G)$ of $G$.
Namely, $\F(G)$ consists of all pairs $(\Gamma,g)$ where
$\Gamma$ is a finite subgraph of the Cayley graph $\Cay(G)$ of $G$ having $1$ and $g$ as vertices and multiplication is defined by the same rule as in $\M(G)$.
We prove that $\F(G)$ is universal among the $X$-generated $F$-inverse monoids $F$ where $F/\sigma$ is a morphic image of
$G$.
More precisely, if there is a canonical morphism $\nu\colon G\to F/\sigma$ then there exists a canonical morphism
$\phi\colon \F(G)\to F$ such that the diagram
\begin{center}
 \begin{tikzcd}
 \F(G) \arrow[r, "\phi"] \arrow[d] & F \arrow[d]\\
 G\arrow[r, "\nu"] & F/\sigma
 \end{tikzcd}
\end{center}
commutes.
As a consequence, we obtain that the free $F$-inverse monoid on $X$ is $\F(G)$ with $G$ being the free group on $X$.
In the language of categories, the main result says that the functor $\F\colon G\mapsto \F(G)$ from the category of $X$-generated groups to the category of $X$-generated $F$-inverse monoids is a left adjoint to the functor $\sigma\colon F\mapsto F/\sigma$ going the other way around.
In particular, $\F(G)$ is an initial object in the category of all $X$-generated $F$-inverse monoids $F$ with $F/\sigma=G$.

In Section \ref{sect:perF} we consider the class of perfect $F$-inverse monoids which are defined by the property that
$\sigma$ is a perfect congruence, that is, the set product of any two $\sigma$-classes is an entire $\sigma$-class.
The collection of all perfect $F$-inverse monoids forms a subvariety and is defined, within the variety of $F$-inverse monoids, by the identity $\mx{x}\mx{(x^{-1})}=1$.
We prove the analogue of the main result within this subclass.
For an $X$-generated group $G$, the universal object $\P(G)$ turns out to be the semidirect product
$\FSL_{G\times X}\rtimes G$ of the free semilattice monoid on $G\times X$ by $G$.
Obviously, this monoid depends only on the group $G$ and the cardinality of $X$, but not on $G$ {\sl as an $X$-generated group.} This is in marked contrast to the case of $\F(G)$ which depends on the geometry of the Cayley graph of $G$, shown in Section \ref{sect:F}.

In Section \ref{sect:F-inverse monoids} we show that the class of all $F$-inverse monoids forms a variety in the enriched signature, provide an identity basis for it and introduce a `convenient' term algebra for this context.
Section \ref{sect:Preliminaries} contains prerequisites.

\section{Preliminaries} \label{sect:Preliminaries}

Here we collect prerequisites concerning inverse monoids, $X$-generated algebraic structures, Cayley graphs of groups, and expansions of groups.

\subsection{Inverse monoids}\label{subsec:inverse monoids}
For an inverse semigroup $S$, $\sigma$ denotes the {\em minimum group congruence} on $S$.
By $E=E(S)$ we denote the semilattice of idempotents of $S$.
An inverse semigroup is {\em $E$-unitary} if $\sigma$ is idempotent pure, that  is, for  any $e\in E$ we have $E\sigma=e\sigma=E$ (no non-idempotent relates to an idempotent under $\sigma$).
By an {\em $F$-inverse monoid} we  mean an inverse monoid where each $\sigma$-class contains a maximum element.
The maximum element of the $\sigma$-class $s\sigma=g\in S/\sigma$ is denoted either $\mx{s}$ or $m_g$  (depending on the point of view), and these elements are referred to as the \emph{max-elements} of $S$.
It is well known that each $F$-inverse monoid is $E$-unitary.

A mapping $\psi\colon G\to S$ from a group $G$ to an inverse monoid $S$ is called a {\em premorphism} if it satisfies $1\psi=1$, $a^{-1}\psi=(a\psi)^{-1}$ and $(ab)\psi \geq a\psi\cdot b\psi$ for any $a,b\in S$.
The notion of a premorphism from a group to an inverse monoid traces back to \cite{MR}, and its connection with partial group actions is studied in \cite{KL}. They are called prehomomorphisms in \cite{Petrich} and dual prehomomorphisms in \cite{Lawson}.
Notice that each premorphism $\psi$ from a group $G$ to an inverse monoid $S$ is {\em strong} in the sense of
Hollings \cite{Hollings}, that is, it satisfies
\begin{equation}\label{def.strong}
g\psi\cdot h\psi=(gh)\psi\cdot (h\psi)^{-1}\cdot h\psi= g\psi\cdot (g\psi)^{-1}\cdot (gh)\psi\quad\hbox{for any}
\quad g,h\in G,
\end{equation}
see \cite[Proposition 2.1, Lemma 2.2]{KL}.
For every $F$-inverse monoid $S$, the mapping $\psi_S\colon S/\sigma\to S$ defined by $g\mapsto m_g$ is a premorphism called
{\em the premorphism induced by $S$}.

\subsection{Categories of $X$-generated algebraic structures}$\,\!$
In this subsection, all algebras are assumed to be of the same algebraic type $\tau$ with a nullary operation.
Let $X$ be a set; an algebra $A$ together with a (not necessarily injective) mapping $\f_A\colon X\to A$ is \emph{$X$-generated} provided that $A$ is generated by $X\f_A$. A \emph{morphism $\psi\colon A\to B$ of $X$-generated algebras} $A$ and $B$ is a morphism of algebras which is \emph{respecting the generators}, that is, satisfies $\f_A\psi=\f_B$.
Such a morphism is called \emph{canonical}. Every canonical morphism is surjective and there is at most one such morphism from any $A$ to any $B$. The class of all $X$-generated algebras naturally forms a category with at most one morphism between any pair of its objects. This category admits an initial object, namely the $X$-generated term algebra $\T_X$ (of type $\tau$). The isomorphism classes of the objects of this category are in bijective correspondence with the congruences $\theta$ on $\T_X$: every quotient algebra $\T_X/\theta$ is an $X$-generated algebra in an obvious way. Moreover, $\T_X/\theta$ is an initial object of the subcategory formed by all algebras which are quotients of $\T_X/\theta$. This is in particular the case when $\theta=\theta(\mathbf{V})$, the fully invariant congruence on $\T_X$ corresponding to some variety $\mathbf{V}$ of  algebras (of type $\tau$).

Throughout, when considering an $X$-generated algebra $A$, the assignment mapping $\f_A$ will be assumed to be understood and will not be explicitly mentioned, except in certain cases when this seems to be important. Given a term $w\in \T_X$ and an $X$-generated algebra $A$, the \emph{value of $w$ in $A$}, that is, the image of $w$ under the canonical morphism $\T_X\to A$ will be denoted by $[w]_A$; for letters $x\in X$ we shall also use the  notation $x_A$ instead of $[x]_A=x\f_A$.

When considering inverse monoids (in particular, groups), $\T_X$ is usually replaced by
the \emph{free monoid with involution} $\I_X$ on $X$. A model of $\I_X$ can be given as follows. Let
$(X\sqcup X^{-1})^*$ be
the free monoid on
$X\sqcup X^{-1}$ where $X^{-1}=\{x^{-1}\colon x\in X\}$ is a (disjoint)
copy of $X$ consisting of formal inverses of the elements of $X$. Setting $(x^{-1})^{-1}=x$ for every $x\in X$, we  get an involution (a bijection of order two)
$X\sqcup X^{-1}\to X\sqcup X^{-1}$,
denoted $y\mapsto y^{-1}$ for any
$y\in X\sqcup X^{-1}$, which can be extended to an
involution on $(X\sqcup X^{-1})^*$  by setting
$1^{-1}=1$ and
$(u_1\cdots u_n)^{-1}=u_n^{-1}\cdots u_1^{-1}$ for any
$n\in \mathbb{N}$
and
$u_i\in X\sqcup X^{-1}$.

\subsection{Cayley graphs of groups}
Let $G$ be an $X$-generated group.
The {\em Cayley graph $\Cay(G)$ of $G$} is the oriented graph $\Cay(G):=V\sqcup E^+\sqcup E^-$ with set of vertices $V=G$, set of positive edges $E^+=G\times X$ and set of negative edges $E^-=G\times X^{-1}$; for the edge $(g,y)\in E:=E^+\cup E^-$ we set $\alpha(g,y)=g$ and $(g,y)\omega=gy_G$ which denote the \emph{initial vertex} and the \emph{terminal vertex} of the edge $(g,y)$, respectively, and $\ell(g,y)=y$ which is called the \emph{label} of  $(g,y)$. Moreover, we have the involution ${}^{-1}\colon E\to E$, defined by $(g,y)^{-1}=(gy_G,y^{-1})$. The edge $(g,y)$ should be thought of as  $\underset{g}{\bullet}\!\overset{y}{\rbl}\!\!\!\underset{gy_G}{\bullet}$; intuitively, the inverse $(g,y)^{-1}$ is `the same edge, but traversed in the opposite direction'. Obviously, the Cayley graph $\Cay(G)$ depends on the assignment mapping
$\f\colon X\to G$; usually this mapping is fixed and assumed to be understood. If necessary, it will be mentioned as here: $\Cay(G,\f)$.

A \emph{non-empty path} in $\Cay(G)$ is a sequence $e_1e_2\cdots e_n$  ($n\ge 1$) of
edges for which $e_i\omega=\alpha e_{i+1}$ for all $i\in \{1,\dots, n-1\}$. For the path $p=e_1\cdots e_n$ we set $\alpha p=\alpha e_1$, $p\omega =e_n\omega$ (initial and terminal vertices of $p$) and $p^{-1}=e_n^{-1}\cdots e_1^{-1}$ (inverse of $p$). We also consider, for each vertex $g$, the \emph{empty path at $g$}, denoted $\varepsilon_g$, for which we set $\alpha\varepsilon_g=g=\varepsilon_g\omega$.
For vertices $g,h\in G$, the path $p$ is a \emph{$(g,h)$-path} if $\alpha p=g$ and $p\omega=h$; paths having the same initial and terminal vertices are called \emph{co-terminal}. The collection of all paths in $\Cay(G)$ forms a (small) category, the \emph{free category} $\Cay(G)^*$ generated by the graph $\Cay(G)$ whose set of objects is the set of vertices $G$ of $\Cay(G)$ while for any $g,h\in G$, the set of all arrows $g\to h$ is comprised of all $(g,h)$-paths, and composition of arrows is the obvious composition of composable paths.

The \emph{label} of the path $p=e_1\cdots e_n$ is defined by $\ell(p)=\ell(e_1)\cdots \ell(e_n)\in \I_X$ while $\ell(\varepsilon_g)=1$ for all vertices $g$; it is immediate that $\ell(p^{-1})= \ell(p)^{-1}$ for every path $p$. Every pair $(w,g)\in \I_X\times G$  admits a unique path in $\Cay(G)$, denoted $p_g(w)$, with $\alpha p_g(w)=g$ and $\ell(p_g(w))=w$; we set $p(w):=p_1(w)$. For the terminal vertex of that path we have $p_g(w)\omega=g[w]_G$, and in particular, $p(w)\omega=[w]_G$.

A \emph{subgraph} $\Gamma$ of $\Cay(G)$ is any subset of $\Cay(G)$ which is closed under $\alpha, \omega$ and $ {}^{-1}$. In particular, every subset $K\subseteq \Cay(G)$ admits a unique subgraph $\left<K\right>$ \emph{spanned} by $K$ which is finite for finite $K$.  The \emph{subgraph spanned by a
path $p$}, denoted $\left<p\right>$, is the subgraph spanned by its edge set
provided $p$ is non-empty, and consists of the vertex $g$ if $p=\varepsilon_g$.
Note that $p$, $p^{-1}$ and $pp^{-1}p$ span the same subgraph. A subgraph $\Gamma$ is \emph{connected} if for any two vertices $u,v\in \Gamma$ there exists a $(u,v)$-path $p$ which runs entirely in $\Gamma$, that is $\left<p\right>\subseteq \Gamma$.

The union $\Gamma\cup \Gamma'$ of two (finite) subgraphs $\Gamma$ and $\Gamma'$ of $\Cay(G)$ is a (finite) subgraph, hence the set of all (finite) subgraphs of $\Cay(G)$ forms a semilattice under union, and so does the set of all (finite) subgraphs  containing the vertex $1$. Left multiplication of the (finite) subgraphs of $\Cay(G)$ by elements of $G$ provides an action of $G$ on the semilattice of (finite) subgraphs of $\Cay(G)$, denoted $(g,\Gamma)\mapsto g\Gamma$.

\subsection{The Margolis--Meakin and the Birget--Rhodes expansions}
Let $G$ be an $X$-generated group and let $\Cay(G)$ be its Cayley graph.
The {\em Margolis--Meakin expansion} $\M(G)$ of $G$ is defined in \cite{MM} as follows.
The elements of $\M(G)$ are the pairs $(\Gamma,g)$ where $\Gamma$ is a finite connected subgraph of $\Cay(G)$ containing $1$  and  $g$  as vertices, and the multiplication on $\M(G)$ is given by
\begin{equation*}
(\Gamma,g)(\Gamma',g') = (\Gamma  \cup g\Gamma', gg').
\end{equation*}
For any $w\in \I_X$, denote by $\Gamma_w$ the subgraph $\left<p(w)\right>$ of $\Cay(G)$ spanned by $p(w)$.
In particular, $\Gamma_1$ has one vertex, $1$, and has no edges, and $\Gamma_x\ (x\in X)$ is the subgraph
whose vertices are $1$ and $x_G$ (which may coincide), and whose only edges are $(1,x)$ and $(1,x)^{-1}=(x_G,x^{-1})$.

\begin{theorem}[\cite{MM}] \label{th:MM}
Let $G$ be an $X$-generated group.
\begin{enumerate}
\item \label{th:MM1} Subject to the mapping $X\to \M(G)$ defined by $x\mapsto (\Gamma_x,x_G)$,
$\M(G)$ is an $X$-generated inverse monoid with  identity element $(\Gamma_1, 1)$ and inverse unary operation $$(\Gamma,g)^{-1}=(g^{-1}\Gamma, g^{-1}).$$
The value  of $w\in \I_X$ in $\M(G)$ is given by $$[w]_{\M(G)}=(\Gamma_w,[w]_G).$$
\item
The idempotents of $\M(G)$ are precisely the elements $(\Gamma,1)\in \M(G)$. Therefore $E(\M(G))$ is isomorphic to the semilattice of all finite connected subgraphs of $\Cay(G)$ containing $1$ as a vertex.
\item
For every $(\Gamma,g),(\Gamma',g')\in \M(G)$, we have $(\Gamma,g)\leq (\Gamma',g')$ if and only if $g=g'$ and
$\Gamma'$ is a subgraph of $\Gamma$.
\item \label{th:MM4}
The projection $\M(G)\to G$ given by $(\Gamma, g) \mapsto g$ is an idempotent pure canonical morphism onto the maximum group quotient $G$ of $\M(G)$. In particular, the inverse monoid $\M(G)$ is $E$-unitary.
\item The Margolis--Meakin expansion $\M(\FG_X)$ of the free $X$-generated group $\FG_X$ is a free $X$-generated inverse monoid.
\item \label{th:MM.univ}
Let $S$ be an $X$-generated $E$-unitary inverse monoid for which there is a canonical morphism $\nu\colon G\to S/\sigma$.
Then there is a canonical morphism $\phi\colon \M(G) \to S$ such that the diagram (of canonical morphisms) \begin{center}
 \begin{tikzcd}
 \M(G) \arrow[r, "\phi"] \arrow[d] & S \arrow[d]\\
 G\arrow[r, "\nu"] & S/\sigma
 \end{tikzcd}
\end{center}
commutes.
\end{enumerate}
\end{theorem}

The Margolis--Meakin expansion $\M(G)$ of the $X$-generated group $G$  is the inverse monoid version of a special case of a general type of expansions (called \emph{Cayley expansions}) which were studied by Elston \cite{elston} and which also appear in the construction of free objects in semidirect product varieties of semigroups and monoids (see Almeida \cite[Section 10]{almeida:book}).

From the definition it seems to be clear that the monoid $\M(G)$ depends (up to isomorphism) not only on the mere group $G$ but rather on $G$ as an $X${-generated group}. This seems to be folklore but the authors did not find in the literature any precise statement --- let alone proof ---  of such an assertion. So, we briefly discuss the question of when for two ${{X}}$-generated groups $G_1$ and $G_2$ we do have $\M(G_1)\cong \M(G_2)$ {as inverse monoids} (rather than as $X$-generated inverse monoids). First of all, if $\M(G_1)\cong \M(G_2)$ as inverse monoids then $G_1=\M(G_1)/\sigma\cong\M(G_2)/\sigma=G_2$ as groups, that is, we may assume that the underlying abstract group is the same in both cases. So, let $G$ be a group, $X_1,X_2$ be sets and $\f_k\colon X_k\to G$ be mappings such that $X_k\f_k$ generates $G$ for $k=1,2$ and denote the resulting $X_k$-generated group by $G_k$.

\begin{proposition}\label{prop:M(G)dependsonCay(G)}
Let $X_1,X_2$, $\f_1,\f_2$ and $G_1, G_2$ be as above; then $\M(G_1)\cong \M(G_2)$ as inverse monoids if and only if there exists a bijection $\beta\colon X_1\to X_2$ such that the Cayley graphs $\Cay(G,\f_1)$ and $\Cay(G,\beta\f_2)$ are isomorphic as $X_1$-labeled \textsl{undirected} graphs.
\end{proposition}

\begin{proof}
First recall from \cite[Lemma 3.3]{MM} that Green's relations $\cJ$ and $\cD$ on $\M(G)$ coincide and that $(\Gamma,g)\mathrel{\cJ}(\Xi,h)$ if and only if $\Gamma\cong\Xi$ (where the isomorphism is understood as isomorphism of edge-labeled graphs). It follows that for any $X$-generated group $G$, in $\M(G)$ the $\cJ$-class of $(\Gamma_x,x_G)$ is a singleton if and only if $x_G=1$ (in which case $\Gamma_x$ is generated by a loop-edge) while if $x_G\ne 1$, the $\cJ$-class of $(\Gamma_x,x_G)$ consists of four elements (for the graph component, there are two possible choices, namely $\left<(1,x)\right>$ and $\left<(x_G^{-1},x)\right>$).

We decompose the set $X_k$ as $X_k=Y_k\sqcup Z_k$ where $y\in Y_k$ if and only if $y\f_k\ne 1$ (that is, we have to distinguish between idempotent and non-idempotent generators of $\M(G)$). We denote the assignment mappings $X_k\to \M(G_k)$ by $\lambda_k$. Recall that for any elements $a,b$ of an inverse semigroup, we have that $J_a\le J_b$ if and only if there exists an element $d$ such that $a\mathrel{\cD} d\le b$ \cite[Proposition 3.2.8]{Lawson}.
Using this, one observes that for $k=1,2$, the $\cJ$-classes of the elements $x\lambda_k$ ($x\in X_k$) are maximal elements below $J_1$ in the partially ordered set $\M(G_k)/\cJ$. Moreover, for $z\in Z_k$ the $\cJ$-class is the singleton $J_{z\lambda_k}$   containing the idempotent $z\lambda_k$  while for $y\in Y_k$, $\vert J_{y\lambda_k}\vert=4$ and $y\lambda_k$ is a non-idempotent member of its $\cJ$-class. It follows that every isomorphism $\phi\colon \M(G_1)\to \M(G_2)$ --- since it induces an order isomorphism $\M(G_1)/\cJ\to \M(G_2)/\cJ$ --- also induces a bijection between the $\cJ$-classes $\{J_{z\lambda_1}\colon z\in Z_1\}$ and $\{J_{z\lambda_2}\colon z\in Z_2\}$ and a bijection between the $\cJ$-classes $\{J_{y\lambda_1}\colon y\in Y_1\}$ and $\{J_{y\lambda_2}\colon y\in Y_2\}$. That is, there is a bijection $\gamma\colon Z_1\to Z_2$ satisfying $J_{z\gamma\lambda_2}=J_{z\lambda_1}\phi$ as well as a bijection $\delta\colon Y_1\to Y_2$ with $J_{y\delta\lambda_2}=J_{y\lambda_1}\phi$. Set $\beta=\gamma\cup \delta$ for the bijection $X_1\to X_2$. Altogether, for every $x\in X_1$ we have $x\lambda_1\phi\in\{x\beta\lambda_2, (x\beta\lambda_2)^{-1}\}$. For convenience of notation, we assume from now on that $X_1=X=X_2$ and $\beta=\mathrm{id}_X$. Then we have $x\lambda_1\phi\in \{x\lambda_2, (x\lambda_2)^{-1}\}$ for all $x\in X$. Next let $X=A\sqcup B$ where $x\in A$ if and only if $x\lambda_1\phi =x\lambda_2$ while $x\in B$ otherwise, and note that in the latter case $x\lambda_1\phi=(x\lambda_2)^{-1}$. Then, for $x\in A$ we have
\begin{equation}\label{variablesA}
(\left<(1,x)\right>,x\f_1)\phi =(\left<(1,x)\right>,x\f_2)
\end{equation}
while for $x\in B$ we have
\begin{align}\label{variablesB}
\begin{split}(\left<(1,x)\right>,x\f_1)\phi&=(\left<(1,x)\right>,x\f_2)^{-1}=((x\f_2)^{-1}\left<(1,x)\right>,(x\f_2)^{-1})\\ &= (\langle(1,x^{-1})\rangle,(x\f_2)^{-1}).\end{split}
\end{align}
Let $\f_3\colon X\to G$ be defined by
\[x\f_3=\begin{cases} x\f_2 &\mbox{ if }x\in A\\ (x\f_2)^{-1}&\mbox{ if }x\in B.\end{cases}\]
Equations (\ref{variablesA}) and (\ref{variablesB}) imply that $(G,\f_1)$ and
$(G,\f_3)$ are isomorphic as $X$-generated groups and therefore have isomorphic Cayley graphs. But the $X$-generated groups $(G,\f_2)$ and $(G,\f_3)$ differ from each other only in that the letters $x$ from $B$ are sent to mutually inverse elements. On the Cayley graphs this has the effect that one can obtain $\Cay(G,\f_2)$ from $\Cay(G,\f_3)$ by reversing all arrows which carry a label from $B$. Consequently, since $\Cay(G,\f_1)$ and $\Cay(G,\f_3)$ are isomorphic as $X$-labeled directed graphs, $\Cay(G,\f_1)$ and $\Cay(G,\f_2)$ are isomorphic as $X$-labeled undirected graphs.

For the converse we only have to consider the case where $X$ is partitioned into $A\sqcup B$ and we have two assigment functions $\f_1,\f_2\colon X\to G$ such that $\f_1$ and $\f_2$ agree on $A$ but for every $x\in B$, $x\f_2=(x\f_1)^{-1}$. From the discussion above it is clear that $\M(G,\f_1)\cong\M(G,\f_2)$ as inverse monoids.
\end{proof}

Now let $G$ be an arbitrary group. The model $\BR(G)$ of the {\em Birget--Rhodes prefix \mbox{expansion} $\widetilde{G}^{\mathcal{R}}$} of $G$, due to the third author \cite{Szendrei}, consists of
the set of all pairs $(A,g)$ with $A$ being a finite subset of $G$ such that $1,g\in A$ and the multiplication on it
defined by
\begin{equation*}
(A,g)(B,h)=(A\cup gB, gh).
\end{equation*}

\begin{theorem}[\cite{Szendrei}]\label{th:Sz}
Let $G$ be an arbitrary group.
\begin{enumerate}
\item \label{th:Sz1}
The semigroup $\BR(G)$ is an inverse monoid generated by
$\left\{(\left\{1,g\right\},g): g\in G \right\},$ with identity element  $(\{1\},1)$ and inverse unary operation
$(A,g)^{-1} = (g^{-1}A, g^{-1})$ for every $(A,g)\in \BR(G)$.
\item
The idempotents of $\BR(G)$ are precisely the elements $(A,1)\in \BR(G)$.
Therefore $E(\BR(G))$ is isomorphic to the semilattice of all finite subsets of $G$ containing $1$ with respect to the operation of union of subsets.
\item
For every $(A,g),(B,h)\in \BR(G)$, we have $(A,g)\leq (B,h)$ if and only if $g=h$ and $B\subseteq A$.
\item
The projection $\BR(G)\to G$ defined by $(A,g) \mapsto g$ is an idempotent pure and
surjective morphism onto the maximum group quotient $G$ of $\BR(G)$.
Moreover, $(\{1,g\},g)$ is the maximum element of its $\sigma$-class for every $g\in G$, and thus $\BR(G)$ is an $F$-inverse monoid.
\item \label{th:Sz.univ}
Let $S$ be an $F$-inverse monoid
and $\nu\colon G\to S/\sigma$ be any morphism.
Then there is a unique morphism $\phi\colon \BR(G) \to S$ such that
max-elements are mapped to max-elements and the diagram
\begin{center}
 \begin{tikzcd}
 \BR(G) \arrow[r, "\phi"] \arrow[d] & S \arrow[d]\\
 G\arrow[r, "\nu"] & S/\sigma
 \end{tikzcd}
\end{center}
commutes.
\end{enumerate}
\end{theorem}

\section{$F$-inverse monoids in enriched signature}\label{sect:F-inverse monoids}

In a natural way, $F$-inverse monoids can be equipped with an additional unary operation such that the class of all $F$-inverse monoids forms a variety in this enriched signature. In this context, we introduce in this section a `convenient' term algebra for $F$-inverse monoids.

On every $F$-inverse monoid $S$, consider the unary operation $a\mapsto \mx{a}\ (a\in S)$ and call it the \emph{max-operation} on $S$. From now on we consider every $F$-inverse monoid as an algebraic structure $(S; \cdot, {}^{-1}, {}\mx{}, 1)$ of signature $(2,1,1,0)$ and, as usual, subalgebras, morphisms and congruences of $F$-inverse monoids are understood in this context. For convenience, subalgebras of $F$-inverse monoids are called \emph{$F$-inverse submonoids}. In particular, each group is an $F$-inverse monoid with the max-operation being the identity mapping. Moreover, for any $F$-inverse monoid $S$, the maximum group congruence $\sigma$ is a congruence and the natural morphism $S\to S/\sigma$ is a morphism in this stricter sense.

First we characterize $F$-inverse monoids within algebras of signature $(2,1,1,0)$.

\begin{proposition}[\cite{Kinyon}] \label{prop:conditions}
An algebra $(S; \cdot, {}^{-1}, {}\mx{}, 1)$ of signature $(2,1,1,0)$ is
an $F$-inverse monoid if and only if $(S; \cdot, {}^{-1}, 1)$ is an inverse monoid and, in addition, the following conditions hold:
\begin{enumerate}
\item \label{prop:cond1}
$\mx{a}\geq a$ for all $a\in S$,
\item \label{prop:cond2}
$\mx{a}=\mx{(ae)}$ for all $a\in S$ and $e\in E(S)$.
\end{enumerate}
\end{proposition}

\begin{proof}
The `only if' part is immediate: in any $F$-inverse monoid $S$, condition \eqref{prop:cond1}
holds by definition, and \eqref{prop:cond2} holds because $a\mathrel{\sigma} ae$ and thus $\mx{a}=\mx{(ae)}$
for all $a\in S$ and $e\in E(S)$.

Conversely, assume that $(S; \cdot, {}^{-1}, {}\mx{}, 1)$ is an algebra of signature $(2,1,1,0)$ such that
$(S; \cdot, {}^{-1}, 1)$ is an inverse monoid and conditions \eqref{prop:cond1} and \eqref{prop:cond2} hold.
We need to show that $\mx{a}$ is the maximum element in the $\sigma$-class $a\sigma$ for any $a\in S$.
Condition \eqref{prop:cond1} implies that $\mx{a} \mathrel{\sigma} a$ for any $a\in S$.
If $a,b\in S$ with $a\mathrel{\sigma} b$
then $ae=be$ for some $e\in E(S)$.
Hence $\mx{a}=\mx{(ae)}=\mx{(be)}=\mx{b}\geq b$ follow by conditions \eqref{prop:cond2} and \eqref{prop:cond1}.
Thus $\mx{a}$ is indeed the maximum element of the $\sigma$-class of $a$.
\end{proof}

The conditions in Proposition \ref{prop:conditions} can be expressed by identities in the signature of $F$-inverse monoids. As a consequence, we get the following.

\begin{corollary}[\cite{Kinyon}] \label{cor:variety}
The class $\bfFA$ of all $F$-inverse monoids forms a variety. It is defined by any identity basis of the variety of inverse
monoids together with the laws
\begin{enumerate}
\item \label{first identity} $\mx{x}x^{-1}x=x$,
\item \label{second identity} $\mx{(xy^{-1}y)}=\mx{x}.$
\end{enumerate}
\end{corollary}

In the following lemma we record several useful laws satisfied by $F$-inverse monoids.

\begin{lemma} \label{lem:first-properties}
Every $F$-inverse monoid satisfies
the laws
\begin{enumerate}
\item \label{lem:first_prop1}
$(\mx{x})^{-1}=\mx{(x^{-1})}$,
\item \label{lem:first_prop3}
$\mx{x} \mx{y} = \mx{(xy)} (\mx{y})^{-1} \mx{y}=  \mx{x} (\mx{x})^{-1} \mx{(xy)}$,
\item \label{lem:first_prop2}
$\mx{(x_0 \mx{y_1}x_1 \cdots x_{n-1} \mx{y_n}x_n)} = \mx{(x_0y_1x_1\cdots x_{n-1}y_nx_n)}$ for every $n\in \mathbb{N}$.
\end{enumerate}
\end{lemma}

\begin{proof}
Let $S$ be an $F$-inverse monoid and $\psi_S$ the premorphism induced by $S$ (see
Section \ref{subsec:inverse monoids}). Then \eqref{lem:first_prop1} is immediate and  \eqref{lem:first_prop3}  follows
from property
\eqref{def.strong} of the premorphism $\psi_S$.
Finally,  \eqref{lem:first_prop2} follows from the fact that, for any $u_0,\dots,u_n,v_1,\dots,v_n\in S$,
the relations $\mx{v_i} \mathrel{\sigma} v_i\ (i=1,\dots, n)$ imply  $u_0\mx{v_1} \cdots  \mx{v_n}u_n \mathrel{\sigma} u_0v_1\cdots v_nu_n$ since $\sigma$ is a congruence.
\end{proof}

We mention two easy but crucial consequences of these observations. Firstly, it follows from items \eqref{lem:first_prop1}
and \eqref{lem:first_prop2} of Lemma \ref{lem:first-properties} that, in the language of $F$-inverse monoids, we can use terms without $\mx{}$-nested expressions. This allows us to introduce a `convenient' term algebra in this context by considering the free objects of the overvariety of $\mathbf{FI}$ consisting of all algebraic structures $(S;\cdot,{}^{-1},{}\mx{},1)$ where $(S;\cdot,{}^{-1},1)$ is a monoid with involution and the laws (\ref{lem:first_prop1}) and (\ref{lem:first_prop2}) of Lemma \ref{lem:first-properties} are satisfied. More precisely, let $(\Imx_X;\cdot,{}^{-1},{}\mx{},1)$ be the $X$-generated free object in the variety of $(2,1,1,0)$-algebras defined by the following laws  (the arities of the operation symbols $\cdot$, ${}^{-1}$, $\mx{}$, $1$ are, respectively, $2,1,1,0$):
\begin{enumerate}
\item $x\cdot(y\cdot z)=(x\cdot y)\cdot z$, \label{associative}
\item $x\cdot 1= x=1\cdot x$,
\item $(x\cdot y)^{-1}=y^{-1}\cdot x^{-1}$,
\item $(x^{-1})^{-1}=x$,
\item $(\mx{x})^{-1}=\mx{(x^{-1})}$,
\item $\mx{(x\cdot(\mx{y}\cdot z))}=\mx{(x\cdot(y\cdot z))}$. \label{m-nested}
\end{enumerate}

From the discussion so far, it should be obvious that every $F$-inverse monoid in enriched signature satisfies the laws (\ref{associative})--(\ref{m-nested}). Consequently, every $X$-generated $F$-inverse monoid is a quotient of $\Imx_X$. For an $X$-generated $F$-inverse monoid $F$ and $w\in \Imx_X$ we shall denote by $[w]_F$ the image of $w$ under the canonical morphism $\Imx_X\to F$.

Now we discuss some facts concerning the elements of $\Imx_X$.
Because of the associative law (\ref{associative}) we may dissolve all brackets $(\ \cdot\ )$ coming from the application of the binary operation $\cdot$ and express this operation simply by means of concatenation. In particular,  law (\ref{m-nested}) may be written in the more convenient form
\begin{enumerate}
\item[(6*)] $\mx{(x\mx{y} z)}=\mx{(xyz)}$.
\end{enumerate}
Moreover,  for every $n\in \mathbb{N}$, law \eqref{lem:first_prop2} of Lemma \ref{lem:first-properties} is a consequence of the laws (\ref{associative})--(\ref{m-nested}*) above. From this it follows that every element $w$ of $(\Imx_X;\cdot,{}^{-1},{}\mx{},1)$ admits a unique representation
\begin{equation}\label{def:termrepresentation}
w=u_0\mx{v_1}u_1\cdots u_{n-1}\mx{v_n}u_n
\end{equation}
for some $n\in \mathbb{N}_0$, $u_0,\dots,u_n,v_1,\dots, v_n\in \I_X$. In the context of the binary operation $\cdot$, the identity element $1$ will be treated as \textsl{empty symbol}; in particular (some of) the expressions $u_i$ in (\ref{def:termrepresentation}) may be empty. From the laws (1)--(4) one obtains $1=1^{-1}$; however, $1=\mx{1}$ is not a consequence of (1)--(6*) (see below), hence $1\ne \mx{1}$ in $\Imx_X$ and therefore every occurrence of $\mx{1}$ must stay visible in the expression (\ref{def:termrepresentation}). We shall call the elements of $(\Imx_X;\cdot,{}^{-1},{}\mx{},1)$  \emph{terms} and we shall usually assume that terms are represented as in (\ref{def:termrepresentation}). That the representation (\ref{def:termrepresentation}) of terms is indeed unique can be verified by looking at the monoid $(X\sqcup X^{-1}\sqcup \mx{\I_X})^*$ where $\mx{\I_X}$ is the set of symbols $\{\mx{w}: w\in \I_X\}$. That monoid  can be made to an $X$-generated $(2,1,1,0)$-algebra satisfying the laws (\ref{associative})--(\ref{m-nested}*) by (obvious) appropriate definitions of the unary operations $\mx{}$ and ${}^{-1}$. (The details are left to the reader.) This also shows that the law $1=\mx{1}$ is not satisfied in $\Imx_X$.

Secondly, each $F$-inverse monoid $S$ is `positively generated' in the sense that if it is generated, as an inverse monoid, by a subset $A\cup \mx{S}$ then it is generated, as a semigroup, by the same set. Slightly more generally, the following holds.

\begin{lemma}
Each subset of an $F$-inverse monoid $S$ containing $\mx{S}$ and being closed under multiplication is an $F$-inverse submonoid of $S$.
\end{lemma}

\begin{proof}
Let $T$ be such a subset. Identity (\ref{first identity}) of Corollary \ref{cor:variety} implies that for any $t\in T$, $t^{-1}t=\mx{(t^{-1})}t$, hence $t=t\mx{(t^{-1})} t$ so that $t$ is a regular element of $T$. Altogether, $T$ is an inverse subsemigroup of $S$ containing all max-elements of $S$. For every  $t\in T$,  $\mx{t}$ is the maximum element of $t\sigma_S$. Since $t\sigma_T\subseteq t\sigma_S$ it is also the maximum element of $t\sigma_T$ and so $T$ is an $F$-inverse submonoid of $S$.
\end{proof}

\begin{corollary}\label{lem:gen-inv-FinvA}
Let $S$ be an $F$-inverse monoid.
If $A\subseteq S$ is a generating set of $S$ (as an $F$-inverse monoid) then $A\cup \mx{S}$ generates $S$ as a semigroup.
\end{corollary}

\section{The universal $F$-inverse monoid over an $X$-generated group} \label{sect:F}

\subsection{Motivation}
In this section we introduce an $F$-inverse monoid analogue to the Mar\-go\-lis--Meakin expansion of an $X$-generated group, and establish a universal property of this expansion among $X$-generated $F$-inverse monoids. In order to give some intuitive motivation for the construction, let first an $X$-generated group $G$ and its Margolis--Meakin expansion be considered as machines doing some computation on input words $w\in \I_X$. The result of the computation of $G$ on the input $w$ is the move  from the initial state $1$ to state $[w]_G$ (or from an earlier state $g$ to $g[w]_G$). In contrast, $\M(G)$ not only does that computation (in the second component), but (in the first component), outputs some information on how the first computation has been achieved. (In a sense, the first component records which states have been accessed during the computation and which basic commands have been executed in which states, regardless of their order and the number of possible repetitions.) We are seeking a  device  amenable to more complicated input words, namely to those of the form $u_0\mx{v_1}\cdots \mx{v_n}u_n\in \Imx_X$ and such that the outcome is the `most general $F$-inverse monoid over $G$'. While the interpretation of the input words $u_i$ should be the same (in the second component, move the head from state $g$ to state $g[u_i]_G$, in the first component record all edges traversed), how should inputs of the form $\mx{v_i}$ be interpreted? In any $F$-inverse monoid $F$ with $F/\sigma=G$, for any words $v_1,v_2\in \I_X$ we have $[\mx{v_1}]_F=[\mx{v_2}]_F$ provided that $[v_1]_G=[v_2]_G$. This means that the term $\mx{v}$ cannot contain more information for $F$ than is contained in $[v]_G$. It seems to be reasonable that in our device, in the second component the head moves from the latest state $g$ to $g[v]_G$, while in the first component, we cannot record anything except that we have arrived in state $g[v]_G$. It seems clear that, letting all  terms $w\in \Imx_X$ as possible inputs, in the first component we obtain all possible finite subgraphs of $\Cay(G)$ containing   $1$ and the entry $[w]_G$ of the second component as vertices.

\subsection{The model}
Let $G$ be an $X$-generated group. Consider the set $\F(G)$ of all pairs $(\Gamma,g)$ such that $\Gamma$ is a finite subgraph of $\Cay(G)$ containing  $1$ and $g$ as vertices, and
define a multiplication on $\F(G)$ by the rule
\begin{equation}\label{eq:mult}
(\Gamma,g)(\Gamma',g') = (\Gamma  \cup g\Gamma', gg')
\end{equation}
extending that seen with Margolis--Meakin expansions.

For $g\in G$, let $\Delta_g$ be the subgraph of $\Cay(G)$ which has no edge and whose only vertices are $1$ and $g$.
Notice that $\Delta_1=\Gamma_1$, but it will be convenient for us to use both forms.
In the following statement we collect several properties of $\F(G)$ which are easy to verify.

\begin{proposition}\label{prop:prop1}
Let $G$ be an $X$-generated group.
\begin{enumerate}
\item The algebraic structure $\F(G)$ is an inverse monoid, with identity element
$(\Gamma_1, 1)$ and inverse unary operation $(\Gamma,g)^{-1}=(g^{-1}\Gamma,g^{-1})$ for all $(\Gamma,g)\in \F(G)$.
\item The idempotents of $\F(G)$ are precisely the elements $(\Gamma,1)\in \F(G)$.
Therefore $E(\F(G))$ is isomorphic to the semilattice of all finite subgraphs of $\Cay(G)$ containing $1$ as a vertex.
\item For every $(\Gamma,g),(\Gamma',g')\in \F(G)$, we have $(\Gamma,g)\leq (\Gamma',g')$ if and only if $g=g'$ and
$\Gamma'$ is a subgraph of $\Gamma$.
\item The  projection $ \F(G)\to G$ defined by $(\Gamma, g) \mapsto g$ is an idempotent pure morphism onto the maximum group quotient $G$ of $\F(G)$.
\item Moreover, each of $(\Delta_g,g)\ (g\in G)$ is the maximum element of its $\sigma$-class,
and thus $\F(G)$ is an $F$-inverse monoid with max-operation $\mx{(\Gamma,g)}=(\Delta_g,g)$ for all $(\Gamma,g)\in \F(G)$.
\end{enumerate}
\end{proposition}

\begin{remark} \label{MM-Sz-F}
\begin{enumerate}
\item \label{MM-F}
The construction of $\F(G)$ is analogous to that of $\M(G)$, and we have $\M(G)\subseteq \F(G)$.
The latter inclusion is  strict because, for $(\Gamma,g)\in \F(G)$, the subgraph $\Gamma$ of $\Cay(G)$
is not necessarily connected.
Theorem \ref{th:MM}\eqref{th:MM1} implies that $\M(G)$ is the inverse submonoid of $\F(G)$ generated by the set
$\{(\Gamma_x, x_G)\colon x\in X\}$.
\item \label{Sz-F}
The subgraphs of $\Cay(G)$ without edges can be naturally identified with their sets of vertices.
Thus the elements $(\Gamma,g)\in \F(G)$ with $\Gamma$ having no edge form an inverse submonoid
of $\F(G)$ isomorphic to $\BR(G)$, and by Theorem \ref{th:Sz}\eqref{th:Sz1}, it is generated by the set
$\mx{(\F(G))}=\{(\Delta_g,g)\colon g\in G\}$.
\end{enumerate}
\end{remark}

Now we complete the list of the basic properties of $\F(G)$.

\begin{proposition}\label{prop:prop2}
Let $G$ be an $X$-generated group; then $\F(G)$ is an $X$-generated $F$-inverse monoid subject to the mapping $\f_{\F(G)}\colon X\to \F(G)$, $x\mapsto (\Gamma_x,x_G)$.
\end{proposition}

\begin{proof}
Denote by $F$ the $F$-inverse submonoid of $\F(G)$ generated by $X\f_{\F(G)}$.
Note that $X\f_{\F(G)}=X\f_{\M(G)}$, and so Remark \ref{MM-Sz-F}\eqref{MM-F} implies that $F\supseteq \M(G)$.
Hence $F\supseteq\mx{(\M(G))}= \mx{(\F(G))}$ follows, since every $\sigma$-class of $\F(G)$ contains an element
of $\M(G)$.
To conclude the proof, it suffices to show that each element of $\F(G)$ is a product of elements of
$\M(G)$ and of $\mx{(\F(G))}$.

Let $(\Gamma,g)\in \F(G)$, and
consider a connected component $\Xi$ of $\Gamma$ and a vertex $v$ of $\Xi$.
Then $(\Delta_v, v) (v^{-1}\Xi, 1) (\Delta_{v^{-1}}, v^{-1})=
(\Delta_v \cup \Xi \cup v\Delta_{v^{-1}}, 1) = (\Xi \cup \Delta_1, 1)$.
Since $v^{-1}\Xi$ is a connected subgraph of $\Cay(G)$ which has $1$ as a vertex, we have
$(v^{-1}\Xi, 1)\in \M(G)$, whence it follows that $(\Xi \cup \Delta_1,1)\in F$.
Since $(\Gamma, g)=(\Gamma, 1)(\Delta_g,g)$,
and $(\Gamma, 1)$ is the product of all $(\Xi\cup \Delta_1, 1)$ where $\Xi$ runs through the connected
components of $\Gamma$, we obtain that $(\Gamma,g)\in F$.
\end{proof}

Similarly to the Margolis--Meakin expansion $\M(G)$, also $\F(G)$ depends on $G$ as an $X$-generated group. In order to establish this result we shall make use of a description of Green's relations on $\F(G)$ which is analogous to the one for $\M(G)$ (see \cite[Lemma 3.3]{MM}).
\begin{proposition}\label{Greens-relations-F(G)}
For  $(\Gamma,g),(\Xi,h)\in \F(G)$ the following hold:
\begin{enumerate}
\item \label{R} $(\Gamma,g)\mathrel{\cR}(\Xi,h)$ if and only if $\Gamma=\Xi$,
\item $(\Gamma,g)\mathrel{\cL}(\Xi,h)$ if and only if $g^{-1}\Gamma=h^{-1}\Xi$,
\item $(\Gamma,g)\mathrel{\cD}(\Xi,h)$ if and only if $\Gamma=k\Xi$ for some $k\in G$,
\item $\cJ=\cD$.
\end{enumerate}
\end{proposition}

\begin{proof}
The first two assertions are immediate consequences of the characterisations of $\cR$ and $\cL$ on inverse semigroups ($s\mathrel{\cR}t$ if and only if $ss^{-1}=tt^{-1}$ and $s\mathrel{\cL}t$ if and only if $s^{-1}s=t^{-1}t$). So let us suppose that $(\Gamma,g)\mathrel{\cD}(\Xi,h)$. There exists $(\Theta,c)\in\F(G)$ such that $(\Gamma,g)\mathrel{\cR}(\Theta,c)\mathrel{\cL}(\Xi,h)$, that is, $\Gamma=\Theta$ and $c^{-1}\Gamma=c^{-1}\Theta=h^{-1}\Xi$ so that $\Gamma=ch^{-1}\Xi$. For the converse assume that $\Gamma=k\Xi$ for some $k\in G$. From $1\in \Gamma=k\Xi$ it follows that $k^{-1}\in \Xi$ whence $(\Xi,k^{-1})\in \F(G)$. Altogether, we obtain
$$(\Gamma,g)\mathrel{\cR}(\Gamma,1)\mathrel{\cL}(\Xi,k^{-1})\mathrel{\cR}(\Xi,h).$$
Finally assume that $(\Gamma,g)\mathrel{\cJ}(\Xi,h)$; then, for $i=1,2$ there exist $(\Theta_i,a_i),(\Phi_i,b_i)\in F(G)$ such that $(\Theta_1,a_1)(\Gamma,g)(\Phi_1,b_1)=(\Xi,h)$ and $(\Theta_2,a_2)(\Xi,h)(\Phi_2,b_2)=(\Gamma,g)$. By use of the finiteness of the graphs $\Gamma$ and $\Xi$ one obtains that there exists $k\in G$ such that $\Gamma=k\Xi$ which
guarantees by (3) that $(\Gamma,g)\mathrel{\cD}(\Xi,h)$.
\end{proof}

\begin{proposition}\label{prop:F(G)_depends_on_Cay(G)}
Let $G$ be a group, $X_1,X_2$ be any sets.  For $k=1,2$ let $\f_k\colon  X_k\to G$ be such that $X_k\f_k$ generates $G$ and denote the corresponding $X_k$-generated groups $G_k$. Then $\F(G_1)\cong \F(G_2)$ as inverse monoids if and only if  there exists a bijection $\beta\colon X_1\to X_2$ such that $\Cay(G,\f_1)$ and $\Cay(G,\beta\f_2)$ are isomorphic as $X_1$-labeled \textsl{undirected} graphs.
\end{proposition}

\begin{proof}
We argue as in the proof of Proposition \ref{prop:M(G)dependsonCay(G)} and decompose $X_k=Y_k\sqcup Z_k$ such that the elements $y\in Y_k$ represent non-idempotent generators while the elements $z\in Z_k$ represent idempotent generators of $\F(G_k)$. Denote the assignment mappings $X_k\to \F(G_k)$ by $\lambda_k$. We shall exploit that every isomorphism $\varphi\colon \F(G_1)\to \F(G_2)$ induces order isomorphisms $(\F(G_1),{\le})\to (\F(G_2),{\le})$ and $\F(G_1)/\cJ\to\F(G_2)/\cJ$. Moreover $\varphi$ maps $\cJ$-classes onto $\cJ$-classes and, within every $\cJ$-class, $\cR$-classes onto $\cR$-classes. We also employ the connection between the $\cJ$-order and the natural order $\le$ mentioned in the proof of Proposition \ref{prop:M(G)dependsonCay(G)}, see \cite[Proposition 3.2.8]{Lawson}.

Let $\varphi\colon  \F(G_1)\to\F(G_2)$ be an isomorphism. For $k=1,2$ the elements $z\lambda_k$ ($z\in Z_k$) are maximal idempotents strictly below $1$ in $\F(G_k)$ and their $\cJ$-classes are singletons. There are no other singleton $\cJ$-classes in $\F(G_1)$ which are lower neighbours of $J_1$, hence $\varphi$ induces a bijection between $Z_1$ and $Z_2$; more precisely, there is a bijection $\gamma\colon Z_1\to Z_2$ such that $J_{z\gamma\lambda_2}=J_{z\lambda_1}\varphi$ for all $z\in Z_1$.

Now let us consider the generators from $Y_k$. The $\cJ$-classes which are maximal below $J_1$ and which are not singletons are the $\cJ$-classes $J_{(\Delta_g,g)}$ of the max-elements $(\Delta_g,g)$ for $g\ne 1$. Since $\varphi$ induces an order isomorphism $\F(G_1)/\cJ\to \F(G_2)/\cJ$ it maps maximal $\cJ$-classes of $\F(G_1)$ (below $J_1$) to those of $\F(G_2)$. We have to sort out those of them which come from generators $y\f_k$ ($y\in Y_k$). Fix a max-element $(\Delta_g,g)$ in $\F(G_1)$ and let $(\Gamma,h)$ be such that $J_{(\Gamma,h)}<J_{(\Delta_g,g)}$. Then $J_{(\Gamma,h)}$ contains an element below $(\Delta_g,g)$, and there is $(\Xi,g)\in \F(G_1)$ such that $(\Gamma,h)\mathrel{\cJ}(\Xi,g)<(\Delta_g,g)$.
We know that $\Xi$ contains the vertices $1$ and $g$. In order that $J_{(\Xi,g)}$ be a lower neighbour of $J_{(\Delta_g,g)}$ there are three options:
\begin{enumerate}
\item\label{thirdvertex}$\Xi$ consists of three isolated vertices $1,g$ and $a$,
\item\label{loop} $\Xi$ consists of  $1$ and $g$ and a loop edge (together with its inverse) attached to one of these vertices,
\item\label{generator} $\Xi$ consists of  $1$ and $g$ and an edge (together with its inverse) connecting these vertices.
\end{enumerate}
From item (\ref{R}) of Proposition \ref{Greens-relations-F(G)} it follows that in the first case the $\cR$-class of $(\Xi,g)$ has three elements, namely $(\Xi,1),(\Xi,g)$ and $(\Xi,a)$. In cases (\ref{loop}) and (\ref{generator}) the $\cR$-class of $(\Xi,g)$ is of size two (consisting of $(\Xi,g)$ and $(\Xi,1)$). In case (\ref{loop}), the $\cJ$-class $J_{(\Xi,g)}$ is below a singleton $\cJ$-class $J_{z\f_1}$ for some $z\in Z_1$ which does not happen in case (\ref{generator}).
Moreover case (\ref{generator}) happens exactly when $g\in \{y\f_1,(y\f_1)^{-1}\}$ for some $y\in Y_1$. Altogether we are able to characterize the  $\cJ$-classes $J=J_{y\f_1}$ with $y\in Y_1$ (that is, the $\cJ$-classes containing the generators $y\f_1$) by the following conditions:
\begin{enumerate}
\item $J$ is a lower neighbour of some maximal $\cJ$-class (below $J_1$),
\item every $\cR$-class in $J$ has size two,
\item $J$ is not below a singleton $\cJ$-class (except $J_1$).
\end{enumerate}
Since the isomorphism $\varphi$ respects these conditions it follows that generator $\cJ$-classes are mapped to generator $\cJ$-classes. More precisely, there is a bijection $\delta\colon Y_1\to Y_2$ which satisfies  $J_{{y\delta\lambda_2}}=J_{{y\lambda_1}}\varphi$ for all $y\in Y_1$. Setting $\beta:=\gamma\cup \delta$  we get the desired bijection $\beta\colon X_1\to X_2$ such that $x\lambda_1\varphi\in \{x\beta\lambda_2,(x\beta\lambda_2)^{-1}\}$ for every $x\in X_1$.
From now on we can proceed in exactly the same way as in the proof of Proposition \ref{prop:M(G)dependsonCay(G)}.
\end{proof}

\subsection{Universal property}
In this subsection we formulate and prove a universal property of $\F(G)$ which is our main result.
Recall that an $X$-generated group $G$ is considered as an $X$-generated $F$-inverse monoid subject to the max-operation $\mx{g}=g$ for all $g\in G$.

We first introduce a generalisation of paths in the Cayley graph $\Cay(G)$ induced by words in $\I_X$; the generalisation should be amenable to terms in $\Imx_X$. For vertices
$g,h$
of $\Cay(G)$ we call a sequence $j=(p_1,\ldots,p_n)$ ($n\ge 1$) of paths in $\Cay(G)$ a
\emph{$(g,h)$-journey} provided that
$\alpha p_1=g$ and $p_n\omega=h$,  and we let $\alpha j:=\alpha p_1$ and $j\omega:=p_n\omega$ be the initial and terminal vertices of the journey $j$; a journey $j$ is \emph{circuit} if $\alpha j=j\omega$. 
In case $n=1$, the journey $(p_1)$ is not distinguished from the path $p_1$.
As for paths, two journeys  are \emph{co-terminal} if they have the same initial and terminal vertices.
Among the paths $p_i$ there may occur also empty paths $\varepsilon_g$ at basepoints $g\in G$. The collection of all journeys in $\Cay(G)$ forms the set of arrows of a (small) category with set of objects $G$ and $(g,h)$-arrows the $(g,h)$-journeys,
the category of paths being a subcategory. 
Given an $(f,g)$-journey $j=(p_1,\dots,p_m)$ and a $(g,h)$-journey $k=(q_1,\dots,q_n)$, the composition is defined by
$$jk=(p_1,\dots,p_mq_1,\dots,q_n)$$
where $p_mq_1$ is the usual composition of paths. (Note that it is essential here to distinguish between $(\dots,p_m,q_1,\dots)$ and $(\dots,p_mq_1,\dots)$!) Every journey can be written as a product of paths and `jumps' of the form $(\eps_g,\eps_h)$ for $g,h\in G$.
By the subgraph $\left<j\right>$ of $\Cay(G)$ \emph{spanned by the journey $j=(p_1,\dots,p_n)$} we mean the graph $\bigcup_{i=1}^n \left<p_i\right>$.

Now we assign to every term $w\in \Imx_X$ and  $g\in G$ a $(g,g[w]_G)$-journey $j_g(w)$ as follows: for $w\in \I_X$ we set
\begin{equation*}
j_g(w):= p_g(w),
\end{equation*}
and
\begin{equation*}
j_g(\mx{w}):=(\eps_g,\eps_{g[w]_G}),
\end{equation*} and for arbitrary $u,v\in \Imx_X$,
\begin{equation*}
j_g(uv):=j_g(u)j_{g[u]_G}(v).
\end{equation*} In compact form, this means that for
\begin{equation*}
w=u_0\mx{v_1}u_1\cdots u_{n-1}\mx{v_n}u_n\in \Imx_X
\end{equation*} and $g\in G$  we have
\begin{equation}\label{def:journey(w)}
j_g(w)=\big(p_g(u_0),p_{g[u_0{v_1}]_G}(u_1),\dots,p_{g[u_0{v_1}\cdots u_{n-1}{v_n}]_G}(u_n)\big).
\end{equation}
The interpretation is that, given the term $u_0\mx{v_1}u_1\cdots u_{n-1}\mx{v_n}u_n$, we start the journey at vertex $g$ by traversing the path labeled $u_0$ which ends at $g[u_0]_G$. Next we read $\mx{v_1}$: this  tells us that we jump from $g[u_0]_G$ to $g[u_0]_G[v_1]_G=g[u_0v_1]_G$ where we continue the journey by following the path labeled $u_1$, and so on. Thereby we assign to every pair $(w,g)\in \Imx_X\times G$  a  journey $j_g(w)$.
From the definition it follows for every $w\in \Imx_X$ that
\begin{equation}\label{def:evaluationF(G)}
[w]_{\F(G)}=(\left< j_1(w)\right>, [w]_G).
\end{equation}

In order to formulate and prove the universal property of $\F(G)$ the following lemma is crucial.

\begin{lemma} \label{lem:journey}
Let $G$ be an $X$-generated group and $F$ be an $X$-generated $F$-inverse monoid for which there is a
canonical
morphism $\nu\colon G\to F/\sigma$.
If $s$ and $t$ are terms which induce co-terminal journeys in $\Cay(G)$ such that $\left<j_1(s)\right>\supseteq \left<j_1(t)\right>$ then $[s]_F\le [t]_F$.
\end{lemma}

\begin{proof}
The proof is by induction on the length of $t$, considered as a word in $(X\sqcup X^{-1}\sqcup \mx{\I_X})^*$.
If  $t=\mx{w}\in\mx{\I_X}$ then $[s]_G=[t]_G$ implies $[s]_F\le [t]_F$
(since $F/\sigma$ is a quotient of $G$ and $[\mx{w}]_F$ is the maximum element in its $\sigma$-class).
Hence we obtain the claim also if $t$ is the empty word since the identity $\mx{1}=1$ holds in $F$.

If $t\in X\sqcup X^{-1}$ then the journey $j_1(s)$ contains the edge $e=j_1(t)$ or its inverse.
In the first case, $j_1(s)=jek$ for some journeys $j$ and $k$. If a corresponding factorisation of $s$ is $s=s_1ts_2$ then $[s_1]_G=1=[s_2]_G$
whence $[s_1]_F$ and $[s_2]_F$ are idempotents.
This implies $[s]_F=[s_1]_F[t]_F[s_2]_F\le [t]_F$.
In the second case, $j_1(s)=je^{-1}k$ for some journeys $j$ and $k$. If a corresponding factorisation of $s$ is $s=s_1t^{-1}s_2$ then $[s_1t^{-1}]_G=1=[t^{-1}s_2]_G$ and therefore $[s_1t^{-1}]_F$ and $[t^{-1}s_2]_F$ are idempotents so that $$[s]_F=[s_1t^{-1}tt^{-1}s_2]_F=[s_1t^{-1}]_F[t]_F[t^{-1}s_2]_F\le[t]_F,$$ and the claim is proved for every word $t$ of length at most $1$.

So we may assume that  $t=uv$ with $u,v\in \Imx_X$ such that $u$ and $v$ both have length shorter than $t$. By the induction assumption, the claim is true for all terms $s'$ with $j_1(s')$ co-terminal with $j_1(u)$ such that $\left<j_1(s')\right>\supseteq \left<j_1(u)\right>$ and likewise for $v$.
The journey $j_1(s)$ must meet the vertex $j_1(u)\omega$. Let $s=s_1s_2$ be a factorisation such that $j_1(s_1)\omega=j_1(u)\omega$; then $j_1(ss_2^{-1})$ and $j_1(u)$ as well as $j_1(s_1^{-1}s)$ and $j_1(v)$ are co-terminal (the latter holds since $j_{[u]_G}(s_1^{-1}s)$ and $j_{[u]_G}(v)$ are co-terminal). By the induction assumption, $[ss_2^{-1}]_F\le[u]_F$ and $[s_1^{-1}s]_F\le[v]_F$ which imply
$[s]_F=[ss_2^{-1}]_F[s_1^{-1}s]_F\le [u]_F[v]_F=[uv]_F=[t]_F$.
\end{proof}

\begin{theorem} \label{thm:univF}
Let $G$ be an $X$-generated group and $F$ be an $X$-generated $F$-inverse monoid such that there is a canonical morphism
$\nu\colon G\to F/\sigma$.
Then there is a canonical morphism $\phi\colon \F(G) \to F$ such that the diagram (of canonical morphisms of $X$-generated
$F$-inverse monoids)
\begin{center}
 \begin{tikzcd}
 \F(G) \arrow[r, "\phi"] \arrow[d] & F \arrow[d]\\
 G\arrow[r, "\nu"] & F/\sigma
 \end{tikzcd}
\end{center}
commutes.
\end{theorem}
\begin{proof}
Let  $u,v\in \Imx_X$ be such that $[u]_{\F(G)}=[v]_{\F(G)}$; by (\ref{def:evaluationF(G)}) $[u]_G=[v]_G$ and the induced journeys $j_1(u)$ and $j_1(v)$ span the same subgraph of $\Cay(G)$. From Lemma \ref{lem:journey} it follows  that $[u]_F=[v]_F$. Consequently, the canonical evaluation morphism $\Imx_X\to F$ given by $w\mapsto [w]_F$ factors through $\F(G)$ which yields the canonical morphism $\phi\colon \F(G)\to F$.
\end{proof}

\begin{remark} \label{F-MM-Sz}
The universal properties of the Margolis--Meakin and Birget--Rhodes expansions
(see Theorems \ref{th:MM}(\ref{th:MM.univ}) and \ref{th:Sz}(\ref{th:Sz.univ}))
are consequences of Theorem \ref{thm:univF}.

Let $S$ be an $X$-generated $E$-unitary inverse monoid and $G$ be an $X$-generated group such that there is a
canonical morphism $\nu\colon G\to S/\sigma$.
Applying the standard embedding of $S$ into its permissible hull $C(S)$ which is $F$-inverse
(see \cite[Theorem 1.4.23 and Proposition 7.1.4]{Lawson}),
consider the inverse submonoid $F$ of $C(S)$ generated by $S\cup \mx{S}$.
Then $F$ is an $X$-generated $F$-inverse monoid with $S$ as its $X$-generated inverse submonoid, and we obtain
by Remark \ref{MM-Sz-F}(\ref{MM-F}) that the canonical morphism $\phi\colon \F(G)\to F$ from 
Theorem \ref{thm:univF} restricted to the $X$-generated inverse subsemigroups of $\F(G)$ and $F$ is a canonical morphism
$\M(G)\to S$ having the required properties.

Turning to the universal property of $\BR(G)$, let $S$ be an inverse monoid which is $F$-inverse, and let $G$ be any group and $\nu\colon G\to S/\sigma$ any morphism.
Observe that since $\BR(G)$ is generated by its max-elements, the morphisms involved in this property are uniquely determined by $\nu$, and have their images in the inverse submonoid $F$ of $S$ generated by $\{m_{g\nu}: g\in G\}$.
Here $F$ is obviously $F$-inverse and $\nu$ restricts to a surjective morphism $G\to F/\sigma$, also denoted $\nu$.
Hence we see that when deducing the universal property of $\BR(G)$, it suffices to consider the approriate morphisms onto $F$ rather than into $S$.
By definition, $F$ is $G$-generated with the assignment mapping defined by $g\f_F=m_{g\nu}$, $G$ is $G$-generated with the identity mapping as assigment mapping and $\nu$ is canonical.
This implies by Remark \ref{MM-Sz-F}(\ref{Sz-F}) that the canonical morphism $\phi\colon\F(G)\to F$ from the previous theorem restricted to the inverse submonoid of $\F(G)$ generated by the max-elements is a morphism $\BR(G)\to F$
satisfying the required conditions.
\end{remark}

\begin{remark} Short and direct syntactic proofs of the universal properties of $\M(G)$ as well as $\BR(G)$ can be given by restricting the arguments of the proofs of Lemma \ref{lem:journey} and Theorem \ref{thm:univF} to terms in $(X\sqcup X^{-1})^*$ (for the case $\M(G)$) or terms in $(\mx{\mathbb I_G})^*$ (for the case $\BR(G)$).
\end{remark}

The following is an immediate consequence of Theorem \ref{thm:univF} and provides a model of the free $X$-generated
$F$-inverse monoid.

\begin{corollary}
The algebraic structure $\F(\FG_X)$ is a free $X$-generated $F$-inverse monoid.
\end{corollary}

For a variety $\mathbf H$ of groups let  $\mathbf{Sl}{\malc}\mathbf{H}$ be the class of all $F$-inverse monoids $F$ for which $F/\sigma\in \mathbf{H}$. Readers familiar with the classical definition of the Ma\softl cev product will recognize $\mathbf{Sl}{\malc}\mathbf{H}$ as the Ma\softl cev product (within the class of all $F$-inverse monoids) of the variety of semilattice monoids with the variety $\mathbf{H}$. Then, for any $w\in \I_X$,  $\mathbf{H}\models w=1$ if and only if $\mathbf{Sl}{\malc}\mathbf{H}\models \mx{w}=1$. This follows from the fact that  for an $X$-generated group $G$ and any $w\in\I_X$ we have that $[w]_G=1$ if and only if $[\mx{w}]_{F(G)}=1$. As a consequence, $\mathbf{Sl}{\malc}\mathbf{H}$ is a variety of $F$-inverse monoids and  the free $X$-generated $F$-inverse monoid in $\mathbf{Sl}{\malc}\mathbf{H}$ is $\F(\FH_X)$ where $\FH_X$ is a free $X$-generated member of $\mathbf{H}$. One can generalize this and can consider \emph{presentations} of $F$-inverse monoids, and subject to the appropriate definition one can observe that if $G$ is presented by relations $w_i=1$ ($i\in I$) with $w_i\in \I_X$ then $\F(G)$ is presented by the relations $\mx{w_i}=1$ ($i\in I$).

\subsection{Categorical issues}
For background in category theory the reader is referred to MacLane \cite{MacLane}.
Let $X$ be a set and $\GX$ and $\FX$ be the categories of $X$-generated groups and $X$-generated $F$-inverse monoids, respectively. The assignment
$F\colon \GX\to \FX$ defined by $G\mapsto \F(G)$ gives rise to a functor. Indeed, let $G,H\in \GX$ with $\nu$ being a canonical morphism $\nu\colon G\to H$. Then $g\mapsto g\nu$ and $(g,x)^{\pm 1}\mapsto (g\nu,x)^{\pm 1}$ provide a mapping $\Cay(G)\to \Cay(H)$, denoted $\hat\nu$, which maps (connected) subgraphs (containing $1$) to (connected) subgraphs (containing $1$). (The mapping $\hat\nu$ can be seen as a canonical graph morphism $\Cay(G)\to \Cay(H)$.) Moreover, this mapping commutes with the actions of $G$ and $H$ on their respective Cayley graphs, that is, $(g\Gamma)\hat\nu=(g\nu)(\Gamma\hat\nu)$, so that the mapping $\F(\nu)\colon  \F(G)\to \F(H)$ defined by $(\Gamma,g)\mapsto (\Gamma\hat\nu,g\nu)$ is a canonical morphism making the diagram
\begin{center}
 \begin{tikzcd}
 \F(G) \arrow[r, "\F(\nu)"] \arrow[d] & \F(H) \arrow[d]\\
 G\arrow[r, "\nu"] &H
 \end{tikzcd}
\end{center}
commute.
Morphisms $G\overset{\nu}{\to}H\overset{\mu}{\to} K$ in $\GX$ lift to
$F(G)\overset{F(\nu)}{\to}F(H)\overset{F(\mu)}{\to} F(K)$ in $\FX$ where $\F(\nu)\F(\mu)$ coincides with $\F(\nu\mu)$ since there is only one morphism $\F(G)\to \F(K)$. Hence the following is straightforward.

\begin{proposition}\label{Fexp}
The functor $\F\colon \GX\to \FX$ is an expansion in the sense of Birget and Rhodes \cite{BR}.
\end{proposition}

Let $\sigma\colon \FX\to \GX$ be the functor which assigns to every $X$-generated $F$-inverse monoid $F$ its maximum group quotient $F/\sigma$.
Theorem \ref{thm:univF} tells us that for any $G\in \GX$ and $F\in \FX$, if there is a morphism $G\to F/\sigma$ then there is also a morphism $\F(G)\to F$. Since $F(G)/\sigma\cong G$ the functor
$\sigma\circ\F$ is naturally isomorphic to $\mathrm{id}_{\GX}$; again there is at most one morphism between any two objects.
Thus Theorem \ref{thm:univF} can be reformulated in the language of categories as follows.

\begin{corollary}\label{Fleftadjoint}
The functor $\F\colon \GX\to \FX$ is left adjoint to the functor $\sigma\colon  \FX\to \GX$.
\end{corollary}

The special case $F/\sigma\cong G$ of Theorem \ref{thm:univF} implies the following.

\begin{corollary}\label{Finitial}
For every $X$-generated group $G$, $\F(G)$ is the initial object in the category of all $X$-generated $F$-inverse monoids $F$ for which $F/\sigma\cong G$.
\end{corollary}

\section{The variety of perfect $F$-inverse monoids} \label{sect:perF}

Following the terminology of Jones \cite{Jones} introduced in the wider context of restriction semigroups, an inverse monoid is called {\em perfect} if it is $F$-inverse and $\sigma$ is a perfect congruence on it, meaning that
the set product of any two $\sigma$-classes is again a whole class.
The same class of restriction semigroups was independently introduced and investigated by the second author in \cite{Kud} under the name \emph{ultra $F$-inverse}.
It is easy to see that an $F$-inverse monoid $S$ is perfect if and only if
$\mx{a}\mx{b} = \mx{(ab)}$ for all $a,b\in S$ or, equivalently, if the premorphism $\psi_S$
induced by $S$ is a morphism.
In other words, the class of all perfect $F$-inverse monoids forms a subvariety $\bfPF$ of $\bfFA$ defined by the law $\mx{x}\mx{y}=\mx{(xy)}$ which is equivalent to $\mx{x}\mx{(x^{-1})}=1$. Indeed, the latter follows from the former by substituting ${x^{-1}}$ for $y$ while the latter and  identity (\ref{lem:first_prop3}) of Lemma \ref{lem:first-properties} imply the former.

Our first application of the main result of Section \ref{sect:F} is to describe universal perfect $F$-inverse monoids.
In order to do so it is sufficient to describe the smallest congruence $\theta$ on $\F(G)$ for which $\F(G)/\theta$ is perfect, that is, we are looking for
a description of the fully invariant congruence $\theta(\bfPF)$ corresponding to the law $\mx{x}\mx{(x^{-1})}=1$.

So, let first $\theta$ be any congruence on $\F(G)$ for which $\F(G)/\theta$ is perfect. Let $(\Gamma,g)\in \F(G)$ and $h\in G$. Since $(\Delta_h,h)$ and $(\Delta_{h^{-1}},h^{-1})=(\Delta_h,h)^{-1}$ are max-elements, we get
\begin{equation*}
(\Delta_h,1)=(\Delta_{h},h)(\Delta_{h^{-1}},h^{-1})\mathrel{\theta} 1
\end{equation*}
and hence
\begin{equation}\label{eq:perfect1}
(\Delta_h\cup\Gamma,g)=(\Delta_h,1)(\Gamma,g)\mathrel{\theta}(\Gamma,g).
\end{equation}
So, if $h$ is not a vertex of $\Gamma$ then the effect of multiplying $(\Gamma,g)$ on the left by $(\Delta_h,1)$ is that the vertex $h$ is added to the graph in the first component and that the resulting element is $\theta$-related to the former.
Hence, reading (\ref{eq:perfect1}) forward and backward, we see that, starting with some element of $F(G)$ we get only $\theta$-related elements if we successively add/remove isolated vertices to/from the first component of the element. For a subgraph $\Gamma$ of $\Cay(G)$ denote by $\Ed(\Gamma)$ the subgraph generated by the edges of $\Gamma$. We have obtained
that
\begin{equation}\label{eq:perfect2}
(\Gamma,g)\mathrel{\theta}(\Xi,g) \mbox{ provided that } \Ed(\Gamma)=\Ed(\Xi).
\end{equation}
This enables us to describe the congruence $\theta(\bfPF)$ on $\F(G)$.

\begin{theorem}\label{thm:perfcong}
For every $X$-generated group $G$, the smallest congruence on $\F(G)$ with perfect quotient is given by the relation
\[(\Gamma,g)\mathrel{\theta(\bfPF)}(\Xi,h)\Longleftrightarrow g=h \mbox{ and\/ } \Ed(\Gamma)=\Ed(\Xi).\]
\end{theorem}

\begin{proof}
It is routine to verify that $\theta(\bfPF)$ is a congruence on $\F(G)$ respecting the max-operation.
By (\ref{eq:perfect2}), $\theta(\bfPF)$ is contained in every congruence with perfect quotient.  Finally, for $g\in G$ we have \[(\Delta_g,g)(\Delta_{g^{-1}},g^{-1})=(\Delta_g,1)\mathrel{\theta(\bfPF)}1\] which implies that $\F(G)/\theta(\bfPF)$ satisfies the law $\mx{x}\mx{(x^{-1})}=1$.
\end{proof}

Observe that $\Ed(\Gamma)$ is the subgraph of $\Gamma$ obtained by removing all isolated vertices of $\Gamma$. This allows us to present a concrete model of the structure $\F(G)/\theta(\bfPF)$.

Let $\Sub_{\Ed}(\Cay(G))$ be the semilattice monoid of \emph{all finite subgraphs of $\Cay(G)$ without isolated vertices} under union and with identity element the empty graph $\varnothing$. The group $G$ acts on $\Sub_{\Ed}(\Cay(G))$ on the left so that we may form the semidirect product $\Sub_{\Ed}(\Cay(G))\rtimes G$ which is an inverse monoid denoted  $P(G)$. We  list some properties of $\P(G)$ analogous to those of $\F(G)$.

\begin{proposition}\label{prop:prop1a}
Let $G$ be an $X$-generated group.
\begin{enumerate}
\item The algebraic structure $\P(G)$ is an inverse monoid with
identity element
$(\varnothing, 1)$, and inverse unary operation $(\Gamma,g)^{-1}=(g^{-1}\Gamma, g^{-1})$ for every $(\Gamma,g)\in \P(G)$.
\item The idempotents of $\P(G)$ are precisely the elements $(\Gamma,1)\in \P(G)$.
Therefore $E(\P(G))$ is isomorphic to the  semilattice monoid
$(\Sub_{\Ed}(\Cay(G));\cup,\varnothing)$.
\item For every $(\Gamma,g),(\Gamma',g')\in \P(G)$, we have $(\Gamma,g)\leq (\Gamma',g')$ if and only if $g=g'$ and
$\Gamma'$ is a subgraph of $\Gamma$.
\item The projection $\P(G)\to G$ defined by $(\Gamma, g) \mapsto g$ is an idempotent pure morphism onto the maximum group quotient $G$ of $P(G)$.
Moreover, each of $(\varnothing,g)\ (g\in G)$ is the maximum element of its $\sigma$-class and thus $\P(G)$ is a perfect $F$-inverse monoid with max-operation $\mx{(\Gamma,g)}=(\varnothing,g)$ for all $(\Gamma,g)\in \P(G)$.
\end{enumerate}
\end{proposition}

The perfect $F$-inverse monoid $\P(G)$ is $X$-generated subject to the mapping $X\to P(G)$ given by $x\mapsto (\Gamma_x,x_G)$, and provides a model of $\F(G)/\theta(\bfPF)$.

\begin{theorem}
The mapping $\F(G)\to \P(G)$ defined by $(\Gamma,g)\mapsto (\Ed(\Gamma),g)$ is a canonical morphism of $X$-generated $F$-inverse monoids which induces the congruence $\theta(\bfPF)$. Consequently, $\P(G)$ is universal for all perfect $X$-generated $F$-inverse monoids $F$ for which there is a canonical morphism $\nu\colon G\to F/\sigma$. That is, there exists a canonical morphism $\phi\colon P(G)\to F$ such that the diagram
\begin{center}
 \begin{tikzcd}
 \P(G) \arrow[r, "\phi"] \arrow[d] & F \arrow[d]\\
 G\arrow[r, "\nu"] & F/\sigma
 \end{tikzcd}
\end{center} commutes.
\end{theorem}

\begin{proof}
The statement about the morphism $\F(G)\to \P(G)$ follows from Theorem \ref{thm:perfcong} and the remarks preceding the statement of the theorem, in particular, $\P(G)\cong \F(G)/\theta(\bfPF)$. Let $F$ be an $X$-generated perfect $F$-inverse monoid for which there exists a canonical morphism $G\to F/\sigma$. Universality of $\F(G)$ implies
that there is a canonical morphism $\F(G)\to F$ (Theorem \ref{thm:univF}).
Since $F$ is perfect, the latter must factor through $\P(G)$ yielding the morphism $P(G)\to F$.
\end{proof}

\begin{corollary}
$\P(\FG_X)$ is a free $X$-generated perfect $F$-inverse monoid.
\end{corollary}

It is also easy to see that the analogues of Proposition \ref{Fexp} and Corollaries \ref{Fleftadjoint} and \ref{Finitial} hold in context of perfect $F$-inverse monoids and with $\F(G)$ replaced by $\P(G)$.

In contrast to $\F(G)$ and $\M(G)$, the expansion $\P(G)$ does not depend on the geometry of the Cayley graph of $G$ as an $X$-generated group. Indeed, every subgraph $\Gamma$ of $\Cay(G)$ without isolated vertices is uniquely determined by $E^+(\Gamma)$, the \textsl{set} of positive edges $(g,x)$ of $\Gamma$. Consequently, the semilattice monoid $\Sub_{\Ed}(\Cay(G))$ is isomorphic to the free semilattice monoid $\FSL_{G\times X}$ generated by $G\times X$. Hence $\P(G_1)\cong \P(G_2)$ as inverse monoids whenever $G_1\cong G_2$ as mere groups and $|X_1|=|X_2|$. The dependency on the geometry of $\Cay(G)$ gets lost in the transition from $\F(G)$ to $\F(G)/\theta(\bfPF)\cong \P(G)$ explained in Theorem \ref{thm:perfcong}. To see \textsl{how} this happens, observe that  $\Sub(\Cay(G))$, the \emph{semilattice  of all finite subgraphs of $\Cay(G)$}  of the $X$-generated group $G$, can be seen as the semilattice monoid freely generated by $G\sqcup \left(G\times X\right)$ subject to the relations
\begin{equation}\label{eq:Sub-relations}
g\vee(h,x)=(h,x)\mbox{ if and only if }g=h\mbox{ or }g=hx_G \quad(g,h\in G,\ x\in X)
\end{equation}
where $\vee$ denotes the semilattice operation.
The relations
(\ref{eq:Sub-relations}) show  the dependency of $\Sub(\Cay(G))$ on the geometry of $\Cay(G)$. The transition $\F(G)\mapsto \F(G)/\theta(\bfPF)$ essentially identifies all isolated vertices in $\Sub(\Cay(G))$ with each other, that is, adds to (\ref{eq:Sub-relations}) the relations
\begin{equation}\label{eq:g=h}
g=h\qquad (g,h\in G)
\end{equation}
so that every vertex acts as an identity on every edge.
The outcome is $\Sub_{\Ed}(\Cay(G))$  which, by (\ref{eq:Sub-relations}) and (\ref{eq:g=h}), is (once more) freely generated by $G\times X$.  As such, this semilattice is  \textsl{geometrically meaningless} and does not depend on the structure of  $\Cay(G)$.

\section{Additional remark}

Recently, Billhardt et al.~\cite{Billhardt_et_al} introduced, for a given $X$-generated group $G$ (subject to
$\f_G\colon X\to G$, $X$ being non-empty), an $E$-unitary inverse semigroup $B=B(G,\f_G)$ (denoted $P$ in \cite{Billhardt_et_al}) which possesses a certain universal property with respect to the permissible hulls $C(S)$ of $X$-generated $E$-unitary inverse semigroups $S$ (subject to $\f_S\colon X\to S$).
Moreover, inside $B$ they were able to identify a copy of the semigroup version $\Msg(G):=\M(G)\setminus\{(\Gamma_1,1)\}$ of the Margolis--Meakin expansion $\M(G)$ and copies of the Birget--Rhodes expansion $\BR(G)$ as inverse subsemigroups. What is more, from the universal property of $B$ they were able to deduce, at once, the characteristic universal properties of $\Msg(G)$ as well as of $\BR(G)$.

It follows from \cite{Billhardt_et_al} that the semigroup $B$ is isomorphic to the inverse subsemigroup $\Q(G)$ of $\P(G)$ generated by the set $\{(\Gamma_x\cup g\Gamma_y,gy_G): x,y\in X,\ g\in G\}$, and consists of all elements $(\Gamma,g)$ of
$\P(G)$ where $1$ and $g$ are vertices of $\Gamma$. Consequently, $\Q(G)$ is an inverse subsemigroup also in $\F(G)$.
The universal property formulated in the main result \cite[Theorem 15]{Billhardt_et_al} can be easily deduced from the universal property of $F(G)$ (by restricting to $\Q(G)$ the morphism $\F(G)\to C(S)$ of $F$-inverse monoids which maps $x\f_{F(G)}$ to the image of $x\f_S$ under the standard embedding $S\to C(S)$, see Remark \ref{F-MM-Sz}).

Finally it should be mentioned that the context is somewhat more general in \cite{Billhardt_et_al}:
instead of requiring that $G$ and $S$ be $X$-generated, it is assumed only that arbitrary mappings $\f_G\colon X\to G$
and $\f_S\colon X\to S$ are fixed.
However, our definitions and results could be adjusted to this generality, as well, and
the adjusted argument of the previous paragraph provides a proof of \cite[Theorem 15]{Billhardt_et_al} in full generality.

\section*{Acknowledgments}

The authors are indebted to Michael Kinyon for his talk at ICS 2018 (\cite{Kinyon}), and are thankful to both him and Peter Jones for the discussions during ICS 2018 and for their valuable input later on.

\end{document}